\documentclass[english]{article}
\usepackage{graphicx}
\usepackage[T1]{fontenc}
\usepackage[latin9]{inputenc}
\usepackage{geometry}
\usepackage{float}
\usepackage{amsthm}
\usepackage{amsmath}
\usepackage{amssymb}
\usepackage{graphicx}
\PassOptionsToPackage{normalem}{ulem}
\usepackage{ulem}

\numberwithin{equation}{section}

\makeatletter

\floatstyle{ruled}
\newfloat{algorithm}{tbp}{loa}
\providecommand{\algorithmname}{Algorithm}
\floatname{algorithm}{\protect\algorithmname}

\theoremstyle{plain}
\newtheorem{thm}{\protect\theoremname}[section]
  \theoremstyle{definition}
  \newtheorem{defn}[thm]{\protect\definitionname}
  \theoremstyle{remark}
  \newtheorem{rem}[thm]{\protect\remarkname}
  \theoremstyle{plain}
  \newtheorem{prop}[thm]{\protect\propositionname}
  \theoremstyle{plain}
  \newtheorem{lem}[thm]{\protect\lemmaname}

\usepackage{hyperref}
\title{Optimal Strategy in ``Guess Who?'': Beyond Binary Search}
\author{Mihai Nica} 

\makeatother

\usepackage{babel}
  \providecommand{\definitionname}{Definition}
  \providecommand{\lemmaname}{Lemma}
  \providecommand{\propositionname}{Proposition}
  \providecommand{\remarkname}{Remark}
\providecommand{\theoremname}{Theorem}

\begin{document}

%
%
%

\global\long\def\p{\mathbf{P}}
\global\long\def\q{\mathbf{Q}}
\global\long\def\cov{\mathbf{Cov}}
\global\long\def\var{\mathbf{Var}}
\global\long\def\corr{\mathbf{Corr}}
\global\long\def\e{\mathbf{E}}
\global\long\def\one{\mathtt{1}}

\global\long\def\pp#1{\mathbf{P}\left(#1\right)}
\global\long\def\ee#1{\mathbf{E}\left[#1\right]}
\global\long\def\norm#1{\left\Vert #1\right\Vert }
\global\long\def\abs#1{\left|#1\right|}
\global\long\def\given#1{\left|#1\right.}
\global\long\def\ceil#1{\left\lceil #1\right\rceil }
\global\long\def\floor#1{\left\lfloor #1\right\rfloor }

\global\long\def\bA{\mathbb{A}}
\global\long\def\bB{\mathbb{B}}
\global\long\def\bC{\mathbb{C}}
\global\long\def\bD{\mathbb{D}}
\global\long\def\bE{\mathbb{E}}
\global\long\def\bF{\mathbb{F}}
\global\long\def\bG{\mathbb{G}}
\global\long\def\bH{\mathbb{H}}
\global\long\def\bI{\mathbb{I}}
\global\long\def\bJ{\mathbb{J}}
\global\long\def\bK{\mathbb{K}}
\global\long\def\bL{\mathbb{L}}
\global\long\def\bM{\mathbb{M}}
\global\long\def\bN{\mathbb{N}}
\global\long\def\bO{\mathbb{O}}
\global\long\def\bP{\mathbb{P}}
\global\long\def\bQ{\mathbb{Q}}
\global\long\def\bR{\mathbb{R}}
\global\long\def\bS{\mathbb{S}}
\global\long\def\bT{\mathbb{T}}
\global\long\def\bU{\mathbb{U}}
\global\long\def\bV{\mathbb{V}}
\global\long\def\bW{\mathbb{W}}
\global\long\def\bX{\mathbb{X}}
\global\long\def\bY{\mathbb{Y}}
\global\long\def\bZ{\mathbb{Z}}

\global\long\def\cA{\mathcal{A}}
\global\long\def\cB{\mathcal{B}}
\global\long\def\cC{\mathcal{C}}
\global\long\def\cD{\mathcal{D}}
\global\long\def\cE{\mathcal{E}}
\global\long\def\cF{\mathcal{F}}
\global\long\def\cG{\mathcal{G}}
\global\long\def\cH{\mathcal{H}}
\global\long\def\cI{\mathcal{I}}
\global\long\def\cJ{\mathcal{J}}
\global\long\def\cK{\mathcal{K}}
\global\long\def\cL{\mathcal{L}}
\global\long\def\cM{\mathcal{M}}
\global\long\def\cN{\mathcal{N}}
\global\long\def\cO{\mathcal{O}}
\global\long\def\cP{\mathcal{P}}
\global\long\def\cQ{\mathcal{Q}}
\global\long\def\cR{\mathcal{R}}
\global\long\def\cS{\mathcal{S}}
\global\long\def\cT{\mathcal{T}}
\global\long\def\cU{\mathcal{U}}
\global\long\def\cV{\mathcal{V}}
\global\long\def\cW{\mathcal{W}}
\global\long\def\cX{\mathcal{X}}
\global\long\def\cY{\mathcal{Y}}
\global\long\def\cZ{\mathcal{Z}}

\global\long\def\sA{\mathscr{A}}
\global\long\def\sB{\mathscr{B}}
\global\long\def\sC{\mathscr{C}}
\global\long\def\sD{\mathscr{D}}
\global\long\def\sE{\mathscr{E}}
\global\long\def\sFA{\mathscr{F}}
\global\long\def\sG{\mathscr{G}}
\global\long\def\sH{\mathscr{H}}
\global\long\def\sI{\mathscr{I}}
\global\long\def\sJ{\mathscr{J}}
\global\long\def\sK{\mathscr{K}}
\global\long\def\sL{\mathscr{L}}
\global\long\def\sM{\mathscr{M}}
\global\long\def\sN{\mathscr{N}}
\global\long\def\sO{\mathscr{O}}
\global\long\def\sP{\mathscr{P}}
\global\long\def\sQ{\mathscr{Q}}
\global\long\def\sR{\mathscr{R}}
\global\long\def\sS{\mathscr{S}}
\global\long\def\sT{\mathscr{T}}
\global\long\def\sU{\mathscr{U}}
\global\long\def\sV{\mathscr{V}}
\global\long\def\sW{\mathscr{W}}
\global\long\def\sX{\mathscr{X}}
\global\long\def\sY{\mathscr{Y}}
\global\long\def\sZ{\mathscr{Z}}

\global\long\def\tr{\text{Tr}}
\global\long\def\re{\text{Re}}
\global\long\def\im{\text{Im}}
\global\long\def\supp{\text{supp}}
\global\long\def\sgn{\text{sgn}}
\global\long\def\d{\text{d}}
\global\long\def\dist{\text{dist}}
\global\long\def\span{\text{span}}
\global\long\def\ran{\text{ran}}
\global\long\def\ball{\text{ball}}
\global\long\def\ai{\text{Ai}}
\global\long\def\occ{\text{Occ}}

\global\long\def\To{\Rightarrow}
\global\long\def\half{\frac{1}{2}}
\global\long\def\oo#1{\frac{1}{#1}}

\global\long\def\al{\alpha}
\global\long\def\be{\beta}
\global\long\def\ga{\gamma}
\global\long\def\Ga{\Gamma}
\global\long\def\de{\delta}
\global\long\def\De{\Delta}
\global\long\def\ep{\epsilon}
\global\long\def\ze{\zeta}
\global\long\def\et{\eta}
\global\long\def\th{\theta}
\global\long\def\Th{\Theta}
\global\long\def\ka{\kappa}
\global\long\def\la{\lambda}
\global\long\def\La{\Lambda}
\global\long\def\rh{\rho}
\global\long\def\si{\sigma}
\global\long\def\ta{\tau}
\global\long\def\ph{\phi}
\global\long\def\Ph{\Phi}
\global\long\def\vp{\varphi}
\global\long\def\ch{\chi}
\global\long\def\ps{\psi}
\global\long\def\Ps{\Psi}
\global\long\def\om{\omega}
\global\long\def\Om{\Omega}
\global\long\def\Si{\Sigma}

\global\long\def\dequal{\stackrel{d}{=}}
\global\long\def\pto{\stackrel{\p}{\to}}
\global\long\def\asto{\stackrel{\text{a.s.}}{\to}}
\global\long\def\dto{\stackrel{\text{d.}}{\to}}
\global\long\def\ld{\ldots}
\global\long\def\di{\partial}

\maketitle
\begin{abstract}
``Guess Who?'' is a popular two player game where players ask ``Yes''/``No'' questions to search for their opponent's secret identity
from a pool of possible candidates. This is modeled as a simple stochastic
game. Using this model, the optimal strategy is explicitly found. Contrary to popular belief, performing a binary search is \emph{not} always optimal. Instead, the optimal strategy for the player who trails is to make certain bold plays in an attempt catch up. This is discovered by first analyzing a continuous version of the game where players play indefinitely and the winner is never decided after finitely many rounds.

\end{abstract}

\section{Introduction}

``Guess Who?'' is a zero-sum two player game where players take
turns asking ``Yes'' or ``No'' questions to find their opponent's
secret identity \cite{wiki_Guess_Who}. Each player keeps track of
a finite pool of possible candidates for their opponent's secret identity.
Players alternate turns asking a ``Yes'' or ``No'' question (that
their opponent must answer truthfully) about their opponent's identity
and reduce their pool of possible candidates. For example, if Player
1 asks: ``Does your character have blue eyes?'' and Player 2 answers
``No'', Player 1 eliminates all candidates in his pool that have
blue eyes. Eventually, only one candidate remains in the pool and
this last character must be their opponent's secret identity! The first
player to narrow their pool down to a single character in this way
wins the game.

To model this game mathematically we will make the following assumptions:
\begin{itemize}
\item The secret identity of the opponent is uniformly distributed amongst
all possible candidates. Because of the eliminating nature of the
``Yes''/``No'' questions, this property persists throughout the
game. With this assumption, only the number of remaining characters
in the pool is relevant to the analysis, not the details of which
characters in particular are remaining.
\item For any candidate pool size $n$ and any $1\leq b\leq n-1$, it is
always possible to construct a question for which exactly $b$ candidates
correspond to the ``Yes'' answer. One way to do this is to sort
the candidates alphabetically and ask ``Does your character\textquoteright{}s
name come alphabetically before $\underline{\ \ \ }$?'' where the
name is chosen to be the name which is $b$-th on the list. We use
the terminology a \textbf{``bid of size $b$''} for such a ``Yes''/``No''
question.
\end{itemize}

With these assumptions, ``Guess Who?'' can be modeled as a so called
\emph{simple stochastic game} as defined by Condon \cite{Condon}. The strategy
of the game is in choosing the bid size. Players can balance risk
vs. reward on each turn by varying their bid size. The bid $b=1$
is a risky play: it gives a small chance to win the game immediately
but is more likely to reduce the candidate pool by only 1. In contrast,
a bid of $b=\floor{\half n}$ is the least risky bid. The main result
of this article is to find the optimal bidding strategy for ``Guess Who''
and prove that it is optimal:
\begin{thm}
\label{thm:Main} (Optimal Strategy and Optimal Probabilities for
``Guess Who?'') When it is Player 1's turn, if Player 1 has $n$ candidates in their pool and Player 2 has $m$ candidates in their pool, then Player 1 has the following optimal strategy: \end{thm}
\begin{itemize}
\item \emph{If $n\geq2^{k+1}+1$ while $2^{k}+1\leq m\leq2^{k+1}$ for some
$k\in\bN\cup\{0\}$, then Player 1 is in the weeds and must make a
bold move to catch up! Their optimal play is a bid of $b^{\ast}(n,m)=2^{k}=2^{\floor{\log_{2}(m-1)}}$
and the probability Player 1 wins if both players play optimally is:
\begin{equation}
p^{\star}(n,m)=\frac{2^{k+1}}{n}-\frac{2}{3}\cdot\frac{2{}^{2k+1}+1}{nm}.
\end{equation}
}
\item \emph{If $2^{k}+1\leq n\leq2^{k+1}$ while $m\geq2^{k}+1$ for some
$k\in\bN\cup\{0\}$, then Player 1 has the upper hand and can stay
ahead by making a low risk play! Their optimal play is a bid of $b^{\ast}(n,m)=\floor{\half n}$
and the probability Player 1 wins if both players play optimally is:}
\begin{equation}
p^{\star}(n,m)=1-\frac{2^{k}}{m}+\frac{2}{3}\cdot\frac{2^{2k}+2}{nm}.
\end{equation}

\end{itemize}

The proof of Theorem \ref{thm:Main} is provided in Section \ref{sec:Proof},
and goes by solving a recurrence relation that $p^{\star}(n,m)$ satisfies.
Figure \ref{fig:Optimal} shows a plot of Player 1's probability of
winning as a function of the pool sizes $(n,m)$.


\begin{figure}[htbp]
\begin{center}
\includegraphics[width=0.49\linewidth]{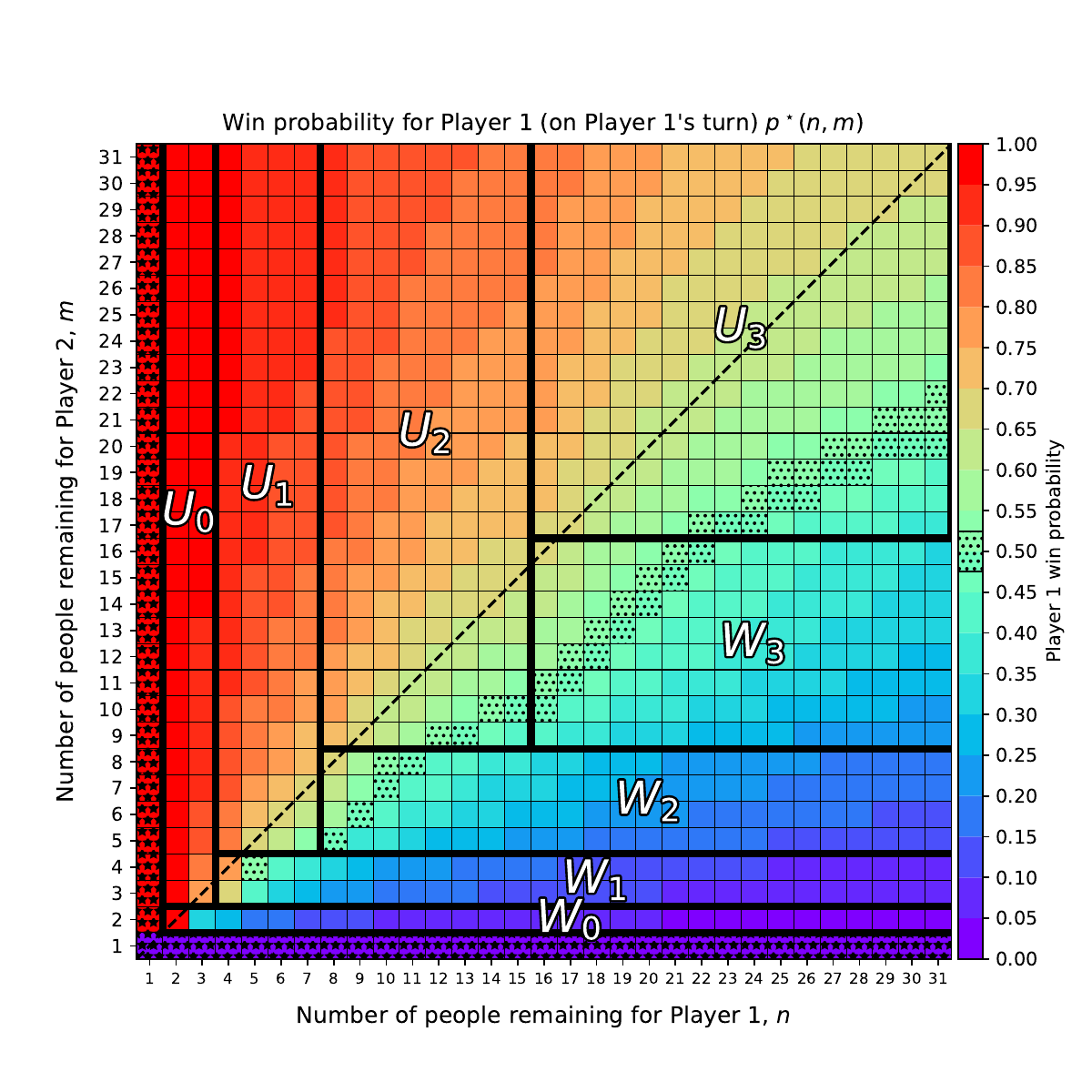}
\includegraphics[width=0.49\linewidth]{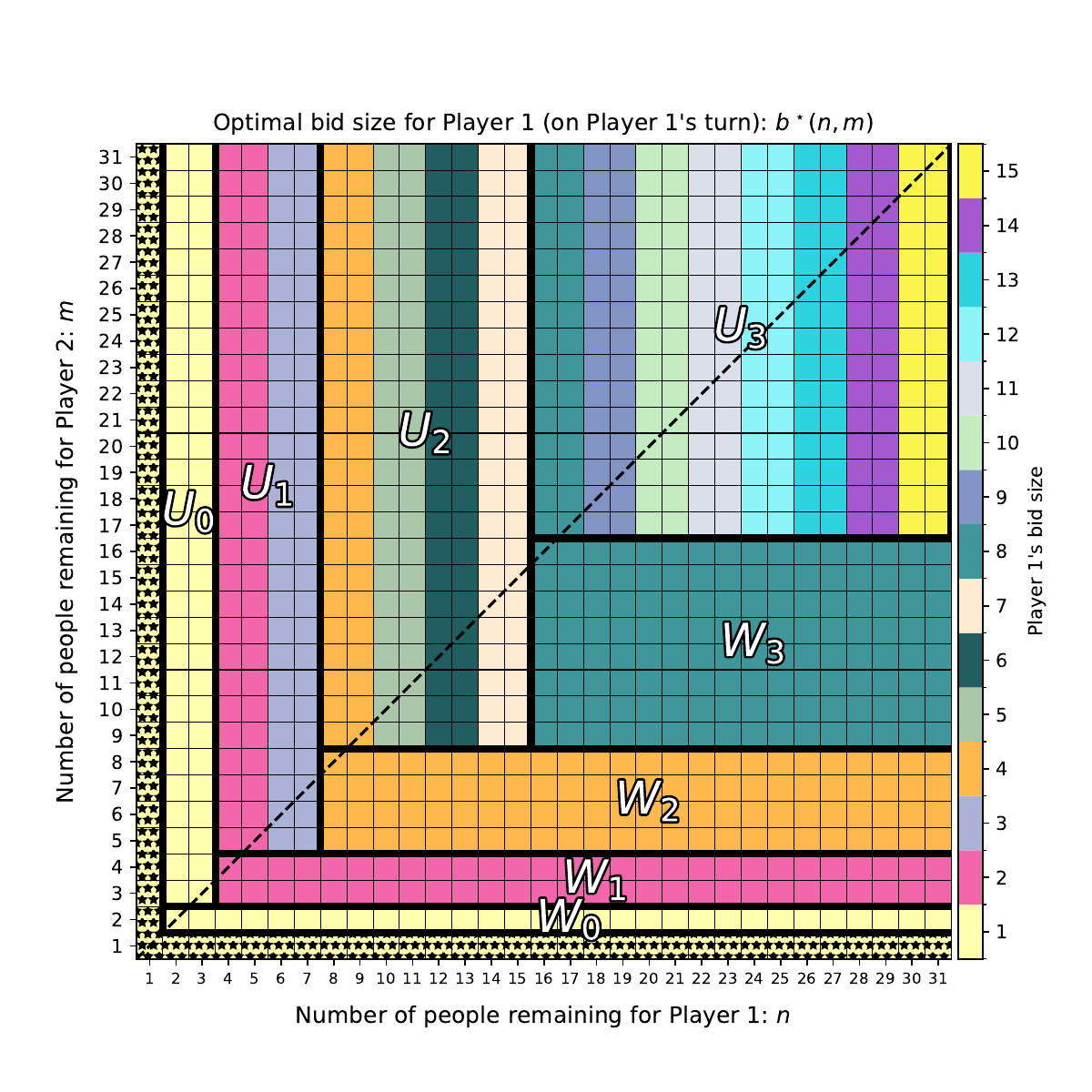}
\end{center}
\caption{\label{fig:Optimal}The winning probability $p^\star(n,m)$ (left) and the optimal bid $b^\star(n,m)$ (right) of Player 1 ``Guess Who?'' on Player 1's turn when Player 1 has $n$ candidates in their pool and Player 2 has $m$ candidates in their pool. The region where $n=1$ or $m=1$ and a player immediatly wins is highlited with stars. Values where $p^\star(n,m)$ is close to 0.5 are also shaded. The ``upper hand'' regions $U_k$ and ``in the weeds'' regions $W_k$ are also labelled.}
\end{figure}

\subsection{Bold Plays}

It is an elementary entropy calculation that the safe bid $b=\floor{\half n}$
will minimize the expected number of ``Yes''/''No'' questions
for Player 1. This is \emph{not }the optimal strategy in ``Guess
Who?'' because Player 1 does not want to minimize this expected value:
instead he wants to maximize the probability of getting there before
Player 2 does. This race against the opponent is what drives the optimal
bidding behavior when Player 1 is significantly behind his opponent.
When this happens, Player 1's optimal strategy is a bold bid, $b=2^{\floor{\log_{2}(m-1)}}$
which depends only on Player 2's (!) remaining pool, and is always
strictly $<\floor{\half n}$. This has a low probability of success
(always strictly$<\half$) but would put Player 1 back in the running
if he were lucky. 

The concept of risky plays with big payoffs in stochastic games has
a rich history. A classic example is the single player casino game
red-and-black when playing against a subfair casino studied by Dubins and Savage \cite{HowToGambleIfYouMust}.
There is also a two player version of red-and-black, where players
try to bankrupt each other, which has also been studied by Secchi \cite{secchi1997}, and adapted by Pontiggia \cite{pontiggia2005}. Still more variations are studied more recently by Chen and Hsiau \cite{chen2010}. In all these examples, the authors find when bold plays are optimal.

Another example, more similar to ``Guess Who?'', is the two player
dice game ``Pig'' where players race to 100 points while managing
risk vs. reward on each of their turns. Without the complication of
racing against the opponent, Haigh and Roters \cite{Haigh_Roters_2000_pig} are able to conduct exact analysis. A game-show version where the players play simultaneously has also been studied by \cite{henk2006_pig}. However,
the interplay between two racing players complicates things. Neller and Presser \cite{PigPractical} have numerically calculated the optimal strategy in situations where both players have $<100$ points using dynamical programming techniques,
but there is no known simple description of the optimal strategy like
we have for ``Guess Who?''. The optimal
strategy to ``Guess Who?'' in Theorem \ref{thm:Main} is satisfying
because we can find a simple exact solution on how and when to play
boldly.

\subsection{Log Periodicity and ``Continuous Guess Who?''}

The landscape of the game exhibits a log periodic behavior. In particular
for fixed $c>0$ , $p^{\star}(n,cn)$ does \emph{not} converge as
$n\to\infty$. Instead it approaches a function which is periodic
in $\log_{2}(n)$. This kind of behavior is not uncommon in this kind
of stochastic system, for example see group Russian roulette studied by van de Brug et al. \cite{GroupRussianRoulette}
or ties amongst i.i.d. random variables studied by Eisenberg et al. \cite{eisenberg1993}. In
``Guess Who?'' there is an asymptotic function $p_{\infty}^{\star}:\bR^{>0}\times\bR^{>0}\to[0,1]$
which is exactly $\log_2$-periodic in the sense that $p_{\infty}^{\ast}(2x,2y)=p_{\infty}^{\star}(x,y)$
and for which $p^{\star}(n,m)=p_{\infty}^{\ast}(n,m)+O\left(\frac{1}{nm}\right)$.
 
This function $p^{\star}_{\infty}$ was first discovered as the probability that Player 1 wins a generalized game called ``Continuous Guess Who?'' which is introduced and studied in Section \ref{sec:Asymptotics}. This game is more straightforward to analyze than ``Guess Who?'' because there are no small number effects. Figure \ref{fig:AsymptoticProb} shows a plot of $p_{\infty}^{\star}(2^{k}\al,2^{k}\be)$
when $(\al,\be)$ is in the L shaped region $(2,\infty)\times(1,2]\bigcup(1,2]\times(1,\infty)$. The landscape of $p_{\infty}^{\star}$ on all of $\bR^{>0}\times\bR^{>0}$ can be understood by tiling scaled copies of this L shaped region onto the entire quadrant.


\begin{figure}[htbp]
\begin{center}
\includegraphics[width=0.95\linewidth]{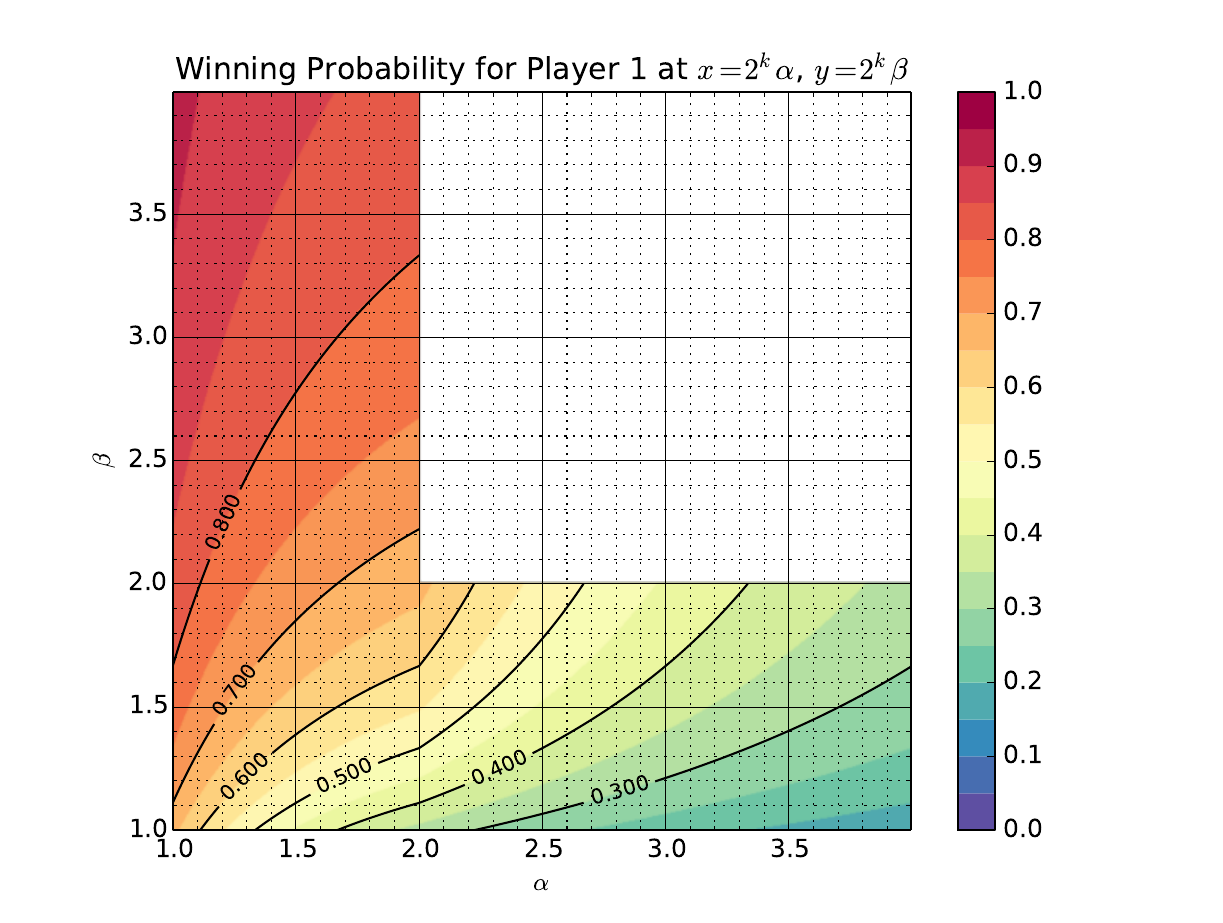}
\end{center}
\caption{\label{fig:AsymptoticProb}The probability $p_{\infty}^{\star}(2^{k}\al,2^{k}\be)$
of Player 1 winning ``Continuous Guess Who?'' from the position $x=2^k\al$, $y=2^k\be$ in the region $(\al,\be)\in(2,\infty)\times(1,2]$ and $(\al,\be)\in(1,2]\times(1,\infty)$. This closely approximates ``Guess Who?'' when $n,m$ are large: $p^{\star}(n,m)=p_{\infty}^{\star}(n,m)+O(\frac{1}{nm})$.}
\end{figure}

This log-periodicity means that if both players start with the same pool size, no matter how large, Player 1 will always have a significant first turn advantage. If he
plays optimally he will win with a probability between $\frac{5}{8}$
and $\frac{2}{3}$, depending on the particular starting size. It
would be more fair to start Player 1 with a larger pool of candidates
to offset his advantage from playing first. The fair way to do this is enlarge his starting pool by a factor
between $\frac{4}{3}$ and $\frac{3}{2}$.

\subsection{Real World Applications}
Even though ``Guess Who?'' is an abstract game, the key features of the game are more broadly applicable:
\begin{itemize}
\item Players are racing to reach a specific goal and get nothing for a second place finish.
\item The progress of both players toward the goal is public knowledge.
\item Players have the ability to balance risk vs reward on the speed of progress toward the goal.
\end{itemize}
In terms of these features, ``Guess Who?'' can be thought of as a toy model of some situations where two firms are competing to be the first to bring a product to market and have the ability to manage risk vs reward on the time it takes for product development. Our analysis suggests that the leading firm should take no risks (i.e. try only to minimize expected development time) and the trailing firm should take the smallest risk possible catch up to the leading firm.

An example from history is the space race between the US and Soviet Union. After Apollo 8 became the first manned mission to orbit the Moon in December 1968, the Soviet Union went forward with a launch of its N1 rocket to attempt lunar orbit in February 1969 despite warnings from engineers on the low probability of mission success. The rocket catastrophically failed, but nevertheless the Soviet Union quickly scrambled for another attempt in July 1969, just 13 days before the launch of Apollo 11. The rocket exploded only 23 seconds after launch and resulted in the largest non-nuclear man-made explosion of all time. Clearly, the Soviet Union was trying to increase the speed of development at the price of increased risk. They knew that taking risks was the best way to maximize the chance of catching up.

\section{The Model}
\begin{defn}
\label{def:MathGuessWho} Mathematician's ``Guess Who'' is a game
between two players played on the statespace: 
\begin{equation}
\bS=\left\{ \left\langle n,m,P_{i}\right\rangle \ :\ n\in\bN,m\in\bN,i\in\{1,2\}\right\} .
\end{equation}
The first entry indicates the number of people remaining in Player
1's pool of possible candidates for Player 2's mystery character.
(This is the number of characters who have not yet been removed from
consideration through previous questions.) The second entry indicates
the same thing for Player 2. The last token, either $P_{1}$ or $P_{2}$
indicates whose turn it is to play. Players alternate turns. 

If the last token is $P_{1}$, it is Player 1's turn. On his turn,
from the state $\left\langle n,m,P_{1}\right\rangle $, Player 1 makes
a \textbf{bid} $b_{1}\in[1,n-1]\cap\bN$. This represents a ``Yes-No''
question that Player 1 may ask Player 2 about his hidden character.
$b_{1}$ represents the size of the set of ``Yes'' answers. Once
Player 1 has selected his bid, the game state evolves stochastically
according to the following rules:
\begin{equation*}
\text{Starting at }\left\langle n,m,P_{1}\right\rangle \text{ with Player 1 bidding }b_{1}\to\begin{cases}
\left\langle b_{1},m,P_{2}\right\rangle  & \text{with probability }\frac{b_{1}}{n}\\
\left\langle n-b_{1},m,P_{2}\right\rangle  & \text{with probability }\frac{n-b_{1}}{n}
\end{cases}.
\end{equation*}
This reflects the fact that Player 2's mystery character is uniformly
distributed in the pool of characters. From this state, it is Player
2's turn. Player 2 makes a bid $b_{2}\in[1,m-1]\cap\bN$, and the
state space evolves in the same way:
\begin{equation*}
\text{From }\left\langle n^{\prime},m,P_{2}\right\rangle \text{ with Player 2 bidding }b_{2}\to\begin{cases}
\left\langle n^{\prime},b,P_{1}\right\rangle  & \text{with probability }\frac{b_{2}}{m}\\
\left\langle n^{\prime},m-b_{2},P_{1}\right\rangle  & \text{with probability }\frac{m-b_{2}}{m}
\end{cases}.
\end{equation*}
All random chances in this game are independent of one another.
Afterward it is Player 1's turn again and play repeats. Players continue
in this way until either player reduces his pool of candidates to
exactly 1. When this happens that player immediately wins.%
\footnote{In some versions of ``Guess Who?'', players are required to make
an additional ``special'' final guess, even if they have already
reduced the pool to a single individual. We do not consider this rule
set here. %
} That is to say that the following states are terminal states for
the game:
\begin{eqnarray*}
\forall m>1,\ \left\langle 1,m,P_{2}\right\rangle  & \leftrightarrow & \text{Player 1 immediately wins}\\
\forall n>1,\ \left\langle n,1,P_{1}\right\rangle  & \leftrightarrow & \text{Player 2 immediately wins}.
\end{eqnarray*}
The state $\left\langle 1,1,P_{i}\right\rangle $ is undefined because
there is no way to reach this state without first passing through
one of the previously defined terminal states.
\end{defn}

\begin{rem}
Definition \ref{def:MathGuessWho} puts the game of ``Guess Who?''
in the framework of a \emph{Simple Stochastic Game} as defined by Condon 
\cite{Condon}. These are also sometimes called $2\half$-player
games, where in addition to Player 1 and Player 2, the randomness
in the game is thought of as $\half$ of a player. Because of the
randomness inherent in these games there is no strategy for either
player that guarantees victory. (Indeed, in ``Guess Who?'', the
opponent can always make a bid of 1, and if they are lucky, will immediately
win.) \end{rem}
\begin{defn}
The theory of simple stochastic games going back to Shapley \cite{Shapley} and Condon \cite{Condon} show the existence of an optimal bidding strategy we will denote
by $b^{\star}:\bN\times\bN\to\bN$ and an optimal probability function
we will denote by $p^{\star}:\bN\times\bN\to[0,1]$ that is optimal
for Player 1 in the sense that%
\footnote{This function $p^{\star}(n,m)$ is sometimes called the \emph{value}
of the node $\left\langle n,m,P_{1}\right\rangle $ and the strategy
$b^{\star}$ is called a \emph{positional} strategy since it depends
only on the current position.%
}:\end{defn}
\begin{itemize}
\item If Player 1 makes a bid of $b^{\star}(n,m)$ whenever the state $\left\langle n,m,P_{1}\right\rangle $
is encountered, then no matter what strategy Player 2 chooses, Player
1 wins with probability $\geq p^{\star}(n,m)$.
\item If Player 1 uses any other bidding strategy whatsoever at the state
$\left\langle n,m,P_{1}\right\rangle $, then Player 2 has a strategy
that ensures that Player 1 wins with probability $\leq p^{\star}(n,m)$.
\end{itemize}

When both players play their optimal strategies, Player 1 will win
with probability exactly $p^{\star}(n,m)$. Since ``Guess Who?''
is symmetric between Player 1 and Player 2 (in the sense that the
position $\left\langle n,m,P_{1}\right\rangle $ is functionally identical
to the position $\left\langle m,n,P_{2}\right\rangle $), the optimal
bidding function for Player 2 from $\left\langle n,m,P_{2}\right\rangle $
is $b^{\star}(m,n)$ and Player 2's optimal probability to win from
$\left\langle n,m,P_{2}\right\rangle $ is $p^{\star}(m,n)$. By the
same token, Player 1's probability to win from $\left\langle n,m,P_{2}\right\rangle $
is $1-p^{\star}(m,n)$ since Player 1 wins if and only if Player 2
loses. Both players strategies and bids can be encapsulated in a single
function. Without loss of generality then, we will thus normally take
the point of view of Player 1 in our analysis. Because of this symmetry,
the optimal probability function satisfies a nice recurrence relation:
\begin{prop}
\label{prop:Recursion}$p^{\star}(n,m)$ and $b^{\star}(n,m)$ satisfy
the following recurrence relation:
\begin{eqnarray}
p^{\star}(n,m) & = & \max_{b\in\left[1,n-1\right]}\left\{1-\frac{b}{n}p^{\star}(m,b)-\frac{n-b}{n}p^{\star}(m,n-b)\right\}\\
b^{\star}(n,m) & = & \arg \max_{b\in\left[1,n-1\right]}\left\{1-\frac{b}{n}p^{\star}(m,b)-\frac{n-b}{n}p^{\star}(m,n-b)\right\}.
\end{eqnarray}
(For the case the argmax is not unique: any $b$ that maximizes the
function will work for Player 1's optimal strategy)\end{prop}
\begin{proof}
Let $p_{b}(n,m)$ be the probability that Player 1 wins from the position
$\left\langle n,m,P_{1}\right\rangle $ if he bids $b$ at $\left\langle n,m,P_{1}\right\rangle $
and both players play optimally thereafter. By the rules of the game
$p_{b}(n,m)=\frac{b}{n}\left(1-p^{\star}(m,n)\right)+\frac{n-b}{n}\left(1-p^{\star}(m,n-b)\right)$
since with probability $\frac{b}{n}$ we move to the position $\left\langle b,m,P_{2}\right\rangle $
where Player 1's probability to win is $1-p^{\star}(m,n)$ and with
probability $\frac{n-b}{n}$ we move to the position $\left\langle n-b,m,P_{2}\right\rangle $
where Player 1's probability to win is $1-p^{\star}(m,n-b)$. 

It must be that $p^{\star}(n,m)=\max_{b\in[1,n-1]}p_{b}(n,m)$ since
Player 1 must bid some value of $b$ at the position $\left\langle n,m,P_{1}\right\rangle $,
and the optimal strategy is no worse than any fixed bid $b$. Similarly,
any $b^{\star}$ that has $p_{b^{\star}}(n,m)=p^{\star}(n,m)$ means that
bidding $b^{\star}$ at the position $\left\langle n,m,P_{1}\right\rangle $
is an optimal bid.\end{proof}
\begin{rem}
``Guess Who?'' has a nice structure because the sum $s=n+m$ is
a \emph{strictly decreasing} function as the game progresses. Indeed,
this sum is guaranteed to decrease by at least 1 on each players turn.
This observation, along with the initial data $\forall m>1, \ p^{\star}(1,m)=1$
and $\forall n>1,\ p^{\star}(n,1)=0$, mean that $p^{\star}(n,m)$,
$b^{\star}(n,m)$ can be explicitly computed inductively by the very
simple Algorithm \ref{alg:Compute} given below%
\footnote{Note that for some $(n,m)$ the value $b^{\star}(n,m)$ is not unique
and will depend on how argmax is implemented. %
}. At each stage $s$, the algorithm has already computed $p^{\star}(n,m)$
for all pairs $(n,m)$ with $n+m\leq s-1$. The values $p^{\star}(m,b)$
and $p^{\star}(m,n-b)$ which appear in the maximization have $m+b\leq s-1$
and $m+(n-b)$ since $1\leq b\leq n-1$ and $m+n=s$, and are hence
already known. The proof of the main Theorem \ref{thm:Main} also
uses this observation as part of a proof by induction. The output
when $r=64$ is displayed in Figure \ref{fig:Optimal} for $1\leq n,m\leq32$. 
\end{rem}


\begin{algorithm}[h]
\emph{Input:} $r\geq3$\emph{ Output:} The value $p^{\star}(n,m)$
and $b^{\star}(n,m)$ for all pairs $(n,m)$ with $3\leq n+m\leq r$

0. Initialize $p^{\star}(1,m)=1$ for all $m>1$ and $p^{\star}(n,1)=0$
for all $n>1$

1. \textbf{for} $s=3$ to $r$:

2. $\ \ $\textbf{for $(n,m)$ }with $n+m=s$:

3. $\ \ \ \ $$p^{\star}(n,m)\mbox{\ensuremath{\leftarrow}}\max_{b\in[1,n-1]}\left\{1-\frac{b}{n}p^{\star}(m,b)-\frac{n-b}{n}p^{\star}(m,n-b)\right\}$

4. $\ \ \ \ $$b^{\star}(n,m)\leftarrow\text{argmax}_{b\in\left[1,n-1\right]}\left\{1-\frac{b}{n}p^{\star}(m,b)-\frac{n-b}{n}p^{\star}(m,n-b)\right\}$

\caption{\label{alg:Compute}Numerically computing $p^{\star}(n,m)$ and $b^{\star}(n,m)$. }
\end{algorithm}

\subsection{``In the weeds'' and ``the upper hand''}
We now divide the statespace $\bS$ into some subsets which turn out
to be relevant for discussing the optimal strategy in the game. Except for the trivial states where $n=1$ or $m=1$, these subsets partition the entire statespace $\bS$ in the sense that each element of the statespace is in exactly one of the subsets. This
division was discovered by careful examination of the output of Algorithm
\ref{alg:Compute}.
\begin{defn}
For $k \in \bN \cup \{0\}$, define the set $W_{k,P_{1}}\subset\bS$ by:
\begin{equation}
W_{k,P_{1}}:=\left\{ \left\langle n,m,P_{1}\right\rangle \ :\ 2^{k+1} < n\text{ and }2^{k} < m\leq2^{k+1}\right\}. 
\end{equation}
When $\left\langle n,m,P_{1}\right\rangle \in W_{k,P_{1}}$, we say
that \textbf{Player 1 is in the weeds at level $k$}, or if $\left\langle n,m,P_{1}\right\rangle \in\bigcup_{k=0}^{\infty}W_{k,P_{1}}$we
say \textbf{Player 1 is in the weeds} without specifying which level.
Similarly, $W_{k,P_{2}}:=$ \newline$\left\{ \left\langle n,m,P_{2}\right\rangle \ :\ 2^{k+1} < m\text{ and }2^{k} <  n\leq2^{k+1}\right\} $
and we use the same terminology for Player 2.
\end{defn}

\begin{defn}
For $k\in \bN \cup \{0\}$, define the set $U_{k,P_{1}}\subset\bS$ by:
\begin{equation}
U_{k,P_{1}}:=\left\{ \left\langle n,m,P_{1}\right\rangle \ :\ 2^{k} < n\leq2^{k+1}\text{ and }2^{k}< m\right\} .
\end{equation}
When $\left\langle n,m,P_{1}\right\rangle \in U_{k,P_{1}}$, we say
that \textbf{Player 1 has the upper hand at level $k$}, or if $\left\langle n,m,P_{1}\right\rangle \in\bigcup_{k=0}^{\infty}U_{k,P_{1}}$we
say \textbf{Player 1 has the upper hand} without specifying which
level. Similarly, $U_{k,P_{2}}:=\left\{ \left\langle n,m,P_{2}\right\rangle \ :\ 2^{k} < m\leq2^{k+1}\text{ and }2^{k} < n\right\} $
and we use the same terminology for Player 2.
\end{defn}

With these two definitions solidified, the description of the optimal
strategy for the game is very simple: players should bid $2^{k}$
when in the weeds at level $k$ and should bid $\floor{\half n}$
when they have the upper hand.

\section{``Continuous Guess Who?''\label{sec:Asymptotics}}

The exact formula in Theorem \ref{thm:Main} is proven by induction,
which does does not shed much light on how the formula is discovered
or what the individual terms in it represent. To aid understanding,
we present the original method that led to the formula $p^{\star}(n,m)$. From numerically computing
$p^{\star}(n,m)$ and $b^{\star}(n,m)$ for some small values of $n,m$
using Algorithm \ref{alg:Compute}, we get an ansatz on the optimal
bidding strategy $b^{\star}(n,m)$. We then exactly calculate the probability for either player to win when they follow this bidding ansatz in a modified version of ``Guess Who?'' which is defined in such a way to remove small number effects. The probability of winning in this modified game turns out to be a good approximation to the probability of winning in ``Guess Who?''  in the asymptotic regime when the pool sizes $n,m$ are very large, and this is what first led to the formula $p^{\star}(n,m)$.

\begin{defn}
``Continuous Guess Who?'' is a two player game on the statespace:
\begin{equation}
\bS_{\infty}=\left\{ \left\langle x,y,P_{i}\right\rangle \ :\ x\in\bR^{> 0},y\in\bR^{> 0},i\in\{1,2\}\right\} .
\end{equation}
The partition of the statespace into  ``in the weeds'' and ``the upper hand'' are generalized to this continuous setting by allowing any exponent $k \in \bZ$:
\begin{eqnarray}
W^{\infty}_{k,P_{1}} &:=& \left\{ \left\langle 2^k \al,2^k \be,P_{1}\right\rangle \ :\  \al \in (2,\infty) \text{ and } \be \in (1,2] \right\}\subset \bS_\infty, k \in \bZ , \\
U^{\infty}_{k,P_{1}} &:=& \left\{ \left\langle 2^k \al,2^k \be,P_{1}\right\rangle \ :\ \al \in (1,2] \text{ and } \be \in (1,\infty) \right\}\subset \bS_\infty , k \in \bZ .
\end{eqnarray}
$W^{\infty}_{k,P_{2}}$, $U^{\infty}_{k,P_{2}}$, $k\in\bZ$ are defined analogously. Play proceeds in a similar way to ``Guess Who?'': on Player 1's turn, from the state $\left\langle x,y,P_{1}\right\rangle $, Player 1 is allowed to make 
\emph{any} (not necessarily an integer) size bid $b_1 \in (0,x)$ and the game state evolves stochastically
according to:
\begin{equation*}
\text{Starting at }\left\langle x,y,P_{1}\right\rangle \text{ with Player 1 bidding }b_{1}\to\begin{cases}
\left\langle b_{1},y,P_{2}\right\rangle  & \text{with probability }\frac{b_{1}}{x}\\
\left\langle x-b_{1},y,P_{2}\right\rangle  & \text{with probability }\frac{x-b_{1}}{x}
\end{cases}.
\end{equation*}
On Player 2's turn he makes a bid $b_2 \in (0,y)$ and the game state evolves analogously. All of these random outcomes are independent.

The game has no ending point: players continue playing indefinitely. Instead, the events $\{ \text{Player 1 loses} \}$, $\{ \text{Player 2 loses} \}$, \{$\text{The game is a draw}\}$ are defined abstractly as events on the underlying probability space. The event that Player 1 loses is defined to be the event that Player 1 is in the weeds ``almost always''; that is there exists some time beyond which Player 1 is trapped in the weeds forevermore. Player 2 losing is defined analogously. The event that the game is a draw is defined to be the event that both players each alternate being in the weeds infinitely often. Denoting the state of the game after $t$ steps as $S_t \in \bS_\infty$, this is written:
\begin{eqnarray*}
\{ \text{Player 1 loses} \} &:=& \left\{ \exists T \text{ such that } \forall t > T, S_t \in \bigcup_{k=-\infty}^{\infty} W^{\infty}_{k,1} \text{ on Player 1's turns} \right\} \\
\{ \text{Player 2 loses} \} &:=& \left\{ \exists T \text{ such that } \forall t > T, S_t \in \bigcup_{k=-\infty}^{\infty} W^{\infty}_{k,2} \text{ on Player 2's turns} \right\} \\
\{ \text{The game is a draw} \} &:=& \left\{ S_t \in \bigcup_{k=-\infty}^{\infty} W^{\infty}_{k,1} \text{  i.o. and } S_t \in \bigcup_{k=-\infty}^{\infty} W^{\infty}_{k,2} \text{ i.o.} \right\}.
\end{eqnarray*}
We define also $\{ \text{Player 2 wins} \} := \{ \text{Player 1 loses} \}$ and $\{ \text{Player 1 wins} \} := \{ \text{Player 2 loses} \}$.

\end{defn}

\begin{rem}
``Continuous Guess Who?'' is a very peculiar game because of the strange way that winning is defined. Indeed, the players cannot know if they have won or lost the game after finitely many rounds of the game. (There is always a positive probability the trailing player will recover and win, no matter how dire their situation may become.) Practically speaking, one should think of ``Continuous Guess Who?'' only as a tool that closely models ordinary ``Guess Who?'' when the length of the game is very long. In the remaining analysis in this section, we will show that when both players play according to an optimal bidding ansatz, the probabilities of winning from any state can be calculated exactly and the probability of a draw is 0. 
\end{rem}

\begin{defn}[Optimal Bidding Ansatz]
We assume both players play according to the following strategy: 
\begin{itemize}
\item When Player 1 is in the weeds at level $k$, $\left\langle 2^k \al, 2^k \be, P_1\right\rangle \in W^{\infty}_{k,P_1}$, Player 1 makes a bid of $2^k$. 
\item When Player 1 has the upper hand at level $k$, $\left\langle x, y, P_1\right\rangle \in U^{\infty}_{k,P_1}$, Player 1 makes a bid of $\frac{1}{2} x$.
\end{itemize}
Player 2 plays analogously. This ansatz was hypothesized by examining output from Algorithm \ref{alg:Compute}.
\end{defn}

\begin{defn}
Define $p_{\infty}^{\star}: \bR^{>0} \times \bR^{>0} \to (0,1)$ to be the probability that Player 1 wins starting from the position $\left\langle x,y,P_1 \right\rangle$ on his turn when both players play according to the optimal bidding ansatz:
\begin{equation}
p_{\infty}^{\star}(x,y) := \p\left(\text{Player 1 wins from }\left\langle x,y,P_{1}\right\rangle \right).
\end{equation}
\end{defn}
\begin{lem}
\label{lem:2scale} For any $j\in\bZ$, $p_{\infty}^{\star}(x,y)$ satisfies:
\begin{equation}
p_{\infty}^{\star}(x,y) = p_{\infty}^{\star}\left(2^{j} x,2^{j} y\right).
\end{equation}
\end{lem}
\begin{proof}
Define the map $\ph: \left\langle x, y, P_i\right\rangle \to \left\langle 2^{j}x , 2^{j}y , P_i\right\rangle$ and we will show that the winning probability $p_{\infty}^{\star}$ is invariant under the map $\ph$. Notice firstly that $\ph$ sends the sets $W^{\infty}_{k,P_i} \to W^{\infty}_{k+j,P_i}$ and $U^{\infty}_{k,P_i} \to U^{\infty}_{k+j,P_i}$. Thus this map has no effect on whether Player 1 is in the weeds or has the upper hand. We next observe that under the optimal bidding ansatz, the probability of a bid being a success is also invariant under the map $\ph$. Indeed, when Player 1 has the upper hand, he makes a no-risk bid and his pool size is plainly halved; when Player 1 is in the weeds at level $k$, $\left\langle 2^k \al, 2^k \be, P_1\right\rangle \in W^{\infty}_{k,P_1}$, $\al > 2$, $\be \in (1,2]$, he makes a bid of size $2^k$ and his pool size after the bid is:
\begin{equation*}
\begin{cases}
2^{k} & \text{ with probability }\frac{2^{k}}{2^{k}\al}\\
2^{k}\al-2^{k} & \text{ with probability }\frac{2^{k}\al-2^{k}}{2^{k}\al}
\end{cases}.
\end{equation*}
The factor $2^k$ cancels, and this probability depends only on $\al$, which is invariant under $\ph$.
 From the above considerations we see that if we let $S_t$ be the state after $t$ rounds started from $ \left\langle x, y, P_1\right\rangle$ and we let $R_t$ to be the state after $t$ rounds started from $ \left\langle 2^{j}x,2^{j}y, P_1\right\rangle $, then we have equality in distribution:
\begin{equation*}
R_t \dequal \ph\left( S_t \right).
\end{equation*}
From this invariance, along with the fact that $\ph$ does not change whether or not Player 1 is in the weeds, and the very definition of Player 1 winning the game, we conclude that $p_{\infty}^{\star}$ is invariant under the map $\ph$ as desired.
\end{proof}

\begin{lem}
\label{lem:EverExit} Assume both players play ``Continuous Guess Who?'' according to the optimal
bidding ansatz. Suppose we start from the position $\left\langle 2^{k}\al,2^{k}\be,P_{1}\right\rangle $
with $\al>2$ and $\be\in(1,2]$, so that Player 1 is in the weeds
at level $k\in\bZ$. Then, the probability that Player 1 \uline{ever}
exits the weeds at any point throughout the game is $\dfrac{2}{\al}$.\end{lem}
\begin{proof}
Since Player 1 is in the weeds at level $k$, his first bid is $2^{k}$.
This means that his pool size becomes:
\begin{equation*}
\begin{cases}
2^{k} & \text{ with probability }\frac{2^{k}}{2^{k}\al}\\
2^{k}\al-2^{k} & \text{ with probability }\frac{2^{k}\al-2^{k}}{2^{k}\al}
\end{cases}.
\end{equation*}
If his pool size becomes $2^{k}$, the pool sizes are now $x=2^{k}$,$y\geq2^{k+1}$
meaning that Player 1 now has escaped the weeds and has the upper
hand. Thus Player 1 escapes from the weeds on the first guess with
probability $\frac{2^{k}}{2^{k}\al}=\frac{1}{\al}$. 

If Player 1 fails to escape the weeds on this turn, his new pool size
is $2^{k}\al-2^{k}=2^{k-1}\left(2(\al-1)\right)$. At this stage Player
2 has the upper hand, and will make a bid of exactly $\half2^{k}\be$
and so Player 2's new pool size is halved and will be exactly $2^{k-1}\be$. Thus,
if Player 1 fails to escape the weeds at this level, the game state
moves to $x=2^{k-1}\left(2(\al-1)\right)$ and $y=2^{k-1}\be$. By Lemma \ref{lem:2scale} the probability of winning is invariant under changing $k$, and so this position is handled by the same analysis
as above with a new $\al$-value which is $\al^{\prime}=2(\al-1)$.
Player 1's probability to escape the weeds at this stage is $\frac{1}{\al^{\prime}}=\frac{1}{2(\al-1)}$.
This analysis can be repeated over and over again. Finally, to find
the probability Player 1 ever escapes the weeds, we sum up his probability
to escape at each stage:
\begin{eqnarray*}
 &  & \p\left(\text{Player 1 ever escapes the weeds}\right)\\
 & = & \frac{1}{\al}+\left(1-\frac{1}{\al}\right)\left(\frac{1}{2(\al-1)}+\left(1-\frac{1}{2(\al-1)}\right)\left(\frac{1}{2\left(2(\al-1)-1\right)}+\ld\right)\right).
\end{eqnarray*}
Fortunately this expansion telescopes and we remain with:
\begin{eqnarray*}
 &  & \p\left(\text{Player 1 ever escapes the weeds}\right)\\
 & = & \frac{1}{\al}+\left(\frac{\al-1}{\al}\right)\frac{1}{2(\al-1)}+\left(\frac{\al-1}{\al}\right)\left(\frac{2(\al-1)-1}{2(\al-1)}\right)\left(\frac{1}{{\scriptstyle 2\left(2(\al-1)-1\right)}}\right)+\ld\\
 & = & \frac{1}{\al}+\frac{1}{2\al}+\frac{1}{4\al}+\ld.
\end{eqnarray*}
Evaluating this infinite sum gives the desired result of $\dfrac{2}{\al}$. 
\end{proof}
Next we prove a lemma about the position $\left\langle 2^{k}\cdot4,2^{k}\cdot2,P_{1}\right\rangle $.
This position is important because it is naturally encountered when
both players play according to the optimal bidding ansatz. 
\begin{lem}
\label{lem:SpecialPos} Assume both players play ``Continuous Guess Who?'' according to the optimal
bidding ansatz. Starting from the position $\left\langle 2^{k}\cdot4,2^{k}\cdot2,P_{1}\right\rangle $
(i.e. $\al=4,\be=2)$, the probability Player 1 loses the game is
$\frac{2}{3}$ and the probability of a draw is $0$.\end{lem}
\begin{proof}
Player 1 can lose the game in one of two disjoint ways:

a) Player 1 never gets out of the weeds. The probability that Player
1 never gets out of the weeds is $1-\frac{2}{4}$ by Lemma \ref{lem:EverExit}. 

b) Player 1 successfully gets out of the weeds at some point and puts
Player 2 in the weeds; then Player 2 himself gets out of the weeds;
and finally Player 1 subsequently loses the game from that
position. The probability of this chain of events is calculated as follows: by Lemma \ref{lem:EverExit}, with probability $\frac{2}{4}$, Player 1 does get out of the weeds at some point, and the position will then
be $x=2^{\ell}\cdot2$ and $y=2^{\ell}\cdot4$ for some $\ell$, (i.e.
the game evolves to $\al^{\prime}=2,\be^{\prime}=4)$. The position must be of this form because Player 2's total is always a power of 2 up to this point (he started at a power of 2 and is always halving his pool up to this point) and Player 1's bid is the largest power of 2 which is strictly less than Player 2's pool size. From this new position, the probability that Player 2 ever gets out of the weeds from this point is $\frac{2}{4}$ by Lemma \ref{lem:EverExit}.
If Player 2 does successfully get out of the weeds, the position
will be $x=2^{j}\cdot4$ and $y=2^{j}\cdot2$ for some $j$ by the same reasoning as above, (i.e.
we the state has evolved back to the original position $\al^{\prime\prime}=4,\be^{\prime\prime}=2$).
Thus we have:
\begin{eqnarray*}
 &  & \p\left(\text{P1 loses from }\left\langle 2^{k}\cdot4,2^{k}\cdot2,P_{1}\right\rangle \right)\\
 & = & \p\left(\text{\small P1 never escapes}\right)+\p\left(\text{\small P1 escapes}\right)\p\left(\text{\small P2 escapes}\right)\p\left(\text{\small P1 loses from }{\scriptstyle \left\langle 2^{j}\cdot4,2^{j}\cdot2,P_{1}\right\rangle }\right)\\
 & = & \left(1-\frac{2}{4}\right)+\frac{2}{4}\cdot\frac{2}{4}\cdot\p\left(\text{P1 loses from }\left\langle 2^{j}\cdot4,2^{j}\cdot2,P_{1}\right\rangle \right)\\
 & = & \half+\oo 4\p\left(\text{P1 loses from }\left\langle 2^{j}\cdot4,2^{j}\cdot2,P_{1}\right\rangle \right).
\end{eqnarray*}

By Lemma \ref{lem:2scale}, the positions $\left\langle 2^{k}\cdot4,2^{k}\cdot2,P_{1}\right\rangle $
and $\left\langle 2^{j}\cdot4,2^{j}\cdot2,P_{1}\right\rangle $ are
identical with regard to losing probability. Solving the linear equation $p=\half+\oo 4p$ gives $p=\frac{2}{3}$,
as desired. The probability of a draw is 0 because the probability of switching from one player being in the weeds to the other being in the weeds from this position has probability $\frac{2}{4}$ each time, and so having infinity many switches is a probability 0 event by the Borel-Cantelli lemma.\end{proof}

\begin{prop}
\label{prop:InTheWeeds} Assume both players play ``Continuous Guess Who?'' according to the optimal
bidding ansatz. Suppose we are at the position $\left\langle x,y,P_{1}\right\rangle =\left\langle 2^{k}\al,2^{k}\be,P_{1}\right\rangle $
with $\al>2$ and $\be\in(1,2]$, so that Player 1 is in the weeds
at level $k\in\bZ$. Then the probability of a draw is $0$ and the probability that Player 1 wins the game is:
\begin{equation}
p_{\infty}^{\star}(x,y) = \frac{2^{k+1}}{x}-\frac{2}{3}\frac{2^{2k+1}}{xy}.
\end{equation}
\end{prop}
\begin{proof}
We first calculate the probability that Player 1 loses. As in Lemma \ref{lem:SpecialPos},
Player 1 can lose the game in one of two disjoint ways:

a) Player 1 never gets out of the weeds. The probability that Player
1 never gets out of the weeds is $1-\frac{2}{\al}$ by Lemma \ref{lem:EverExit}. 

b) Player 1 successfully gets out of the weeds at some point and puts
Player 2 in the weeds; then Player 2 himself gets out of the weeds;
and finally player Player 1 subsequently loses the game from that
position. If Player 1 does get out of the weeds, the position will
be $x=2^{\ell}\cdot2$ and $y=2^{\ell}\cdot\left(2\be\right)$ for
some $\ell$, (i.e. the position evolves to $\al^{\prime}=2,\be^{\prime}=2\be)$.
Thus the probability that Player 2 gets out of the weeds from this
point is $\frac{2}{2\be}$ by Lemma \ref{lem:EverExit}. If Player 2 does successfully get out
of the weeds, the position will be $x=2^{j}\cdot4$ and $y=2^{j}\cdot2$
for some $j$, (i.e. we will be at $\al^{\prime\prime}=4,\be^{\prime\prime}=2$).
We have already analyzed this position in Lemma \ref{lem:SpecialPos}.
Thus we have:
\begin{eqnarray*}
 &  & \p\left(\text{P1 loses from }\left\langle 2^{k}\al,2^{k}\be,P_{1}\right\rangle \right)\\
 & = & \p\left(\text{\small P1 never escapes}\right)+\p\left(\text{\small P1 escapes}\right)\p\left(\text{\small P2 escapes}\right)\p\left(\text{\small P1 loses from }{\scriptstyle \left\langle 2^{j}\cdot4,2^{j}\cdot2,P_{1}\right\rangle }\right)\\
 & = & \left(1-\frac{2}{\al}\right)+\frac{2}{\al}\left(\frac{2}{2\be}\right)\p\left(\text{P1 loses from }\left\langle 2^{j}\cdot4,2^{j}\cdot2,P_{1}\right\rangle \right)\\
 & = & \left(1-\frac{2}{\al}\right)+\frac{2}{\al}\left(\frac{2}{2\be}\right)\frac{2}{3}.
\end{eqnarray*}
The probability of a draw is $0$ since if more than two changes of the upper hand occur, the game passes through the position $\left\langle 2^{j}\cdot4,2^{j}\cdot2,P_{1}\right\rangle$, from which the probability of a draw is 0 by Lemma \ref{lem:SpecialPos}. Finally then, taking the complement and plugging in $x = 2^k\al, y=2^k\be$ gives the result. \end{proof}
\begin{prop}
\label{prop:UpperHand} Assume both players play ``Continuous Guess Who?'' according to the optimal
bidding ansatz. Suppose we are at the position $\left\langle x,y,P_{1}\right\rangle =\left\langle 2^{k}\al,2^{k}\be,P_{1}\right\rangle $
with $\al\in(1,2]$ and $\be>2$, so that Player 1 has the upper hand at level $k\in\bZ$.
Then, the probability of a draw is 0 and the probability that Player 1 wins the game is:
\begin{equation}
p_{\infty}^{\star}(x,y) = 1-\frac{2^{k}}{y}+\frac{2}{3}\frac{2^{2k}}{xy}.
\end{equation}
\end{prop}
\begin{proof}
When Player 1 has the upper hand, he bids exactly $\half 2^{k}\al$
and puts Player 2 in the weeds. By the symmetry between Player 1 and
Player 2, we use Proposition \ref{prop:InTheWeeds} to see the probability of a draw is 0 and to compute:
\begin{eqnarray*}
 &  & \p\left(\text{Player 1 wins from state}\left\langle 2^{k}\al,2^{k}\be,P_{1}\right\rangle \right)\\
 & = & \p\left(\text{Player 1 wins from state }\left\langle 2^{k-1}\al,2^{k-1}\left(2\be\right),P_{2}\right\rangle \right)\\
 & = & 1-\p\left(\text{Player 2 wins from state }\left\langle 2^{k-1}\al,2^{k-1}\left(2\be\right),P_{2}\right\rangle \right)\\
 & = & 1-\p\left(\text{Player 1 wins from state }\left\langle 2^{k-1}\left(2\be\right),2^{k-1}\al,P_{1}\right\rangle \right)\\ 
 & = & 1-\left(\frac{2}{2\be}-\frac{2}{2\be}\left(\frac{2}{2\al}\right)\frac{2}{3}\right).
\end{eqnarray*}
Plugging in $x=2^k\al$,$y=2^k\be$ gives the result.
\end{proof}
\begin{rem}
Despite the difference between ``Continuous Guess Who?'' and ordinary ``Guess Who?'', this analysis comes
remarkably close to the true value $p^{\star}(n,m)$ when evaluated at integer values. Indeed we have:
\begin{eqnarray}
p^{\star}(n,m)-p_{\infty}^{\star}(n,m) & = & \begin{cases}
-\frac{2}{3}\frac{1}{nm} & \text{ if }\left\langle n,m,P_{1}\right\rangle \in W_{k,P_{1}}\\
\frac{4}{3}\frac{1}{nm} & \text{ if }\left\langle n,m,P_{1}\right\rangle \in U_{k,P_{1}}
\end{cases}.
\end{eqnarray}
 The corrections $-\frac{2}{3}\frac{1}{nm}$ and $\frac{4}{3}\frac{1}{nm}$
were found serendipitously by comparing numerical
solutions of $p^{\star}$ to the exact formula $p_{\infty}^{\star}$. The exact formula for $p^{\star}$ was first hypothesized by adding these simple errors onto $p^{\star}_\infty$. A plot of $p_{\infty}^{\star}$ is provided in Figure \ref{fig:AsymptoticProb}.

\end{rem}

\section{Proof of Theorem \ref{thm:Main} \label{sec:Proof}}

Having guessed at the formula using the above analysis, we now
prove Theorem \ref{thm:Main} rigorously. Define functions:
\begin{eqnarray}
q_{W_{k}}(n,m) & := & \frac{2^{k+1}}{n}-\frac{2}{3}\frac{2^{2k+1}+1}{mn}\\
q_{U_{k}}(n,m) & := & 1-\frac{2^{k}}{m}+\frac{2}{3}\frac{2^{2k}+2}{mn}.
\end{eqnarray}
By stitching these functions together, these define the total function
$q(n,m)$:
\begin{equation}
q(n,m)=\begin{cases}
q_{W_{k}}\left(n,m\right)\ \text{if} & \left\langle n,m,P_{1}\right\rangle \in W_{k,P_{1}}\\
q_{U_{k}}\left(n,m\right)\ \text{if} & \left\langle n,m,P_{1}\right\rangle \in U_{k,P_{1}}\\
1 & \text{if }n=1\\
0 & \text{if }m=1
\end{cases}.
\end{equation}
\begin{lem}
\label{lem:qWeeds}Suppose that $\left\langle n,m,P_{1}\right\rangle \in W_{k,P_{1}}$
for some $k$. Then: 
\begin{equation}
q_{W_{k}}(n,m)=\max_{b\in[1,n-1]}\left\{1-\frac{b}{n}q(m,b)-\frac{n-b}{n}q(m,n-b)\right\},
\end{equation}
and the maximum is achieved at $b=2^{k}$.\end{lem}
\begin{proof}
By symmetry between \textbf{$b$ }and\textbf{ $n-b$}, we assume without loss
that $1\leq b\leq\floor{\half n}$. We will show that $1-\frac{b}{n}q(m,b)-\frac{n-b}{n}q(m,n-b)\leq q_{W_{k}}(n,m)$
for all possible bids $b$ with equality achieved at $b=2^{k}$. Consider
the following cases:

\uline{Case I:} $\left\langle b,m,P_{2}\right\rangle \in W_{\ell,P_{2}}$
with $2^{\ell}< b\leq2^{\ell+1}$ and $\left\langle n-b,m,P_{2}\right\rangle \in U_{k,P_{2}}$.

We begin by simplifying:
\begin{eqnarray}
 &  & 1-\frac{b}{n}q_{W_{\ell}}(m,b)-\frac{n-b}{n}q_{U_{k}}(m,n-b) \label{eqn:b_eqn} \\
 & = & 1-\frac{b}{n}\left(\frac{2^{\ell+1}}{m}-\frac{2}{3}\frac{2^{2\ell+1}+1}{mb}\right)-\frac{n-b}{n}\left(1-\frac{2^{k}}{n-b}+\frac{2}{3}\frac{2^{2k}+2}{m(n-b)}\right) \nonumber \\
 & = & \frac{b}{n}\left(1-\frac{2^{\ell+1}}{m}\right)+\frac{2^{k}}{n}+\frac{2}{3}\frac{2^{2\ell+1}-2^{2k}-1}{mn}. \nonumber 
\end{eqnarray}
Since $\left\langle b,m,P_{2}\right\rangle \in W_{\ell,P_{2}}$, we have $2^{\ell+1} < m$ here. It follows that for a fixed $\ell$,
the above is a monotone increasing function of $b$. We thus obtain
an inequality by replacing $b$ by its maximum value $b=2^{\ell+1}$,
with equality if and only if $b=2^{\ell+1}$:
\begin{eqnarray*}
 & & 1-\frac{b}{n}q_{W_{\ell}}(m,b)-\frac{n-b}{n}q_{U_{k}}(m,n-b)\\
 & \leq & \frac{2^{\ell+1}}{n}\left(1-\frac{2^{\ell+1}}{m}\right)+\frac{2^{k}}{n}+\frac{2}{3}\frac{2^{2\ell+1}-2^{2k}-1}{mn}\\
 & = & \frac{2^{\ell+1}}{n}\left(1 -\frac{2}{3}\frac{2^{\ell+1}}{m}\right)+\frac{2^{k}}{n}-\frac{2}{3}\frac{2^{2k}+1}{mn}.
\end{eqnarray*}

This is monotone increasing as a function of $2^{\ell+1}$ since $2^{\ell+1} < m$. To see this, notice that if we increment $\ell$ by one, the difference is:
\begin{equation*}
2^{\ell+1}\left(1-\frac{2}{3m}2^{\ell+1}\right) - 2^{\ell}\left(1-\frac{2}{3m}2^{\ell}\right)=2^{\ell} \left(1 - \frac{2^{\ell+1}}{m} \right) > 0.
\end{equation*}
Thus this function is maximized when $\ell$ takes its maximum value. We notice that we must have $\ell\leq k-1$, since $\left\langle b,m,P_2 \right\rangle \in W_{\ell,P_2}$ gives $2^{\ell+1} < m$ and $\left\langle n,m,P_1 \right\rangle \in W_{k,P_1}$ gives $m\leq 2^{k+1}$.  Thus this function is maximized when $\ell=k-1$, with equality if and only if $\ell=k-1$. By inspection, plugging in $\ell = k-1, b=2^{k}$ gives the maximum value is exactly $q_{W_{k}}(n,m)$ as desired.

\uline{Case II:} $\left\langle b,m,P_{2}\right\rangle \in U_{k,P_{2}}$
and $\left\langle n-b,m,P_{2}\right\rangle \in U_{k,P_{2}}$.

Consider here:
\begin{eqnarray*}
 &  & 1-\frac{b}{n}q_{U_{k}}(m,b)-\frac{n-b}{n}q_{U_{k}}(m,n-b)\\
 & = & 1-\frac{b}{n}\left(1-\frac{2^{k}}{b}+\frac{2}{3}\frac{2^{2k}+2}{mb}\right)-\frac{n-b}{n}\left(1-\frac{2^{k}}{n-b}+\frac{2}{3}\frac{2^{2k}+2}{m(n-b)}\right)\\
 & = & 0+\frac{2^{k}}{n}-\frac{2}{3}\frac{2^{2k}+2}{mn}+\frac{2^{k}}{n}-\frac{2}{3}\frac{2^{2k}+2}{mn}\\
 & = & \frac{2^{k+1}}{n}-\frac{2}{3}\frac{2^{2k+1}+4}{mn}\\
 & < & \frac{2^{k+1}}{n}-\frac{2}{3}\frac{2^{2k+1}+1}{mn}=q_{W_{k}}(n,m).
\end{eqnarray*}
Interestingly, this does not depend on $b$ at all, and is always
worse than $q_{W_k}(n,m)$ by $\frac{2}{mn}$. 

\uline{Case III:} $\left\langle b,m,P_{2} \right\rangle \in W_{\ell_{b},P_{2}}$ and $\left\langle n-b,m,P_{2}\right\rangle\in W_{\ell_{n-b},P_{2}}$
, where $\ell_{b}=\floor{\log_{2}(b-1)}$ $\ell_{n-b}=\floor{\log_{2}(n-b-1)}$.

Assume without loss that $\ell_{n-b}\geq\ell_{b}$. As in Case I, because $\left\langle n-b,m,P_{2} \right\rangle \in W_{\ell_{n-b},P_{2}}$ and $\left\langle n,m,P_{1} \right\rangle \in W_{k,P_{1}}$, we must have $\ell_{b}\leq\ell_{n-b}\leq k-1$.
But then we see that there are no values of $b$ that fall into this
case at all, since we would have $b\leq2^{\ell_{b}+1}\leq2^{k}$ and
$n-b\leq2^{\ell_{n-b}}\leq2^{k}$ which, by summing, is seen to be
in contradiction to the fact that $n\geq2^{k+1}+1$ from $\left\langle n,m,P_{1} \right\rangle \in W_{k,P_{1}}$.\end{proof}
\begin{lem}
\label{lem:qUpperHand}Suppose that $\left\langle n,m,P_{1}\right\rangle \in U_{k,P_{1}}$
for some $k$. Then: 
\begin{equation}
q_{U_{k}}(n,m)=\max_{b\in[1,n-1]}\left\{1-\frac{b}{n}q(m,b)-\frac{n-b}{n}q(m,n-b)\right\},
\end{equation}
and the maximum is achieved at $b=\floor{\frac{1}{2}n}$.\end{lem}
\begin{proof}
By symmetry between $b$ and $n-b$, we assume without loss that $1\leq b\leq\floor{\half n}$.
We will show that $1-\frac{b}{n}q(m,b)-\frac{n-b}{n}q(m,n-b)\leq q_{U_{k}}(n,m)$
for all possible bids $b$ with equality achieved at $b=\floor{\half n}$.
Consider the following cases:

\uline{Case I}: $\left\langle b,m,P_{2}\right\rangle \in W_{\ell,P_{2}}$
with $2^{\ell}< b\leq2^{\ell+1}$ and $\left\langle n-b,m,P_{2}\right\rangle \in U_{k,P_{2}}$.

We start with Eq. \ref{eqn:b_eqn} for $1-\frac{b}{n}q_{W_{\ell}}(m,b)-\frac{n-b}{n}q_{U_{k}}(m,n-b)$ and show directly that the difference from $q_{U_k}(n,m)$ is always positive. It is convenient to begin by multiplying out the common denominator:
\begin{eqnarray*}
 & & nm \left( q_{U_k}(n,m) - \left( 1-\frac{b}{n}q_{W_{\ell}}(m,b)-\frac{n-b}{n}q_{U_{k}}(m,n-b)\right)\right)\\
 &=& nm \left( \left(1-\frac{2^{k}}{m}+\frac{2}{3}\frac{2^{2k}+2}{mn}\right) - \left( \frac{b}{n}\left(1-\frac{2^{\ell+1}}{m}\right)+\frac{2^{k}}{n}+\frac{2}{3}\frac{2^{2\ell+1}-2^{2k}-1}{mn} \right)\right)\\
 &=& nm - n2^k + \frac{2}{3}\left(2^{2k} + 2\right) - b\left(m - 2^{\ell+1}\right) - m 2^k -\frac{2}{3}\left(2^{2\ell+1} - 2^{2k} - 1\right)\\
 &=& n \left(m - 2^k\right) - b\left(m - 2^{\ell+1}\right) - m 2^k -\frac{2}{3}2^{2\ell+1} + \frac{4}{3} 2^{2k} + 2.
\end{eqnarray*}
Now notice that since $\left\langle n-b,m,P_{2}\right\rangle \in U_{k,P_{2}}$, it must be that $n-b > 2^k$ and $m > 2^k$. Thus the coefficient of $n$ is positive, so using $n > 2^{k} + b$ gives:
\begin{eqnarray*}
 &>& \left(2^{k} + b\right) \left(m - 2^k\right) - b\left(m - 2^{\ell+1}\right) - m 2^k -\frac{2}{3}2^{2\ell+1} + \frac{4}{3} 2^{2k} + 2\\
 & = & m 2^{k} + b m - 2^{2k} - b 2^{k} -bm + b2^{\ell+1} - m 2^k -\frac{2}{3}2^{2\ell+1} + \frac{4}{3} 2^{2k} + 2\\
 &=& -b\left(2^k - 2^{\ell+1}\right) - \frac{2}{3}2^{2\ell+1} + \frac{1}{3}2^{2k} + 2.
\end{eqnarray*}
We notice now that it must be the case that $\ell+1 \leq k$, because $\left\langle n-b,m,P_{2}\right\rangle \in U_{k,P_{2}}$ gives $2^{k+1} \geq m$ and $\left\langle b,m,P_{2}\right\rangle \in W_{\ell,P_{2}}$ gives $m > 2^{\ell+1}$. Hence the coefficient of $b$ is non-positive, and we can use $b \leq 2^{\ell+1}$ to get:
\begin{eqnarray*}
 &\geq& -2^{\ell+1}\left(2^k - 2^{\ell+1}\right) - \frac{2}{3}2^{2\ell+1} + \frac{1}{3}2^{2k} + 2\\
 &=& \frac{2}{3}2^{2(\ell+1)} - 2^{\ell+1}2^{k} + \frac{1}{3} 2^{2k} + 2.
\end{eqnarray*}
Finally, since $\ell+1 \leq k$, define $\ell+1 = k - c$, where $c\in \bN \cup\{0\}$ and then write the above expression as:
\begin{equation*}
= 2^{2k}\left( \frac{2}{3}2^{-2c} - 2^{-c} + \frac{1}{3} \right) + 2.
\end{equation*}  
By inspection, $\frac{2}{3}2^{-2c} - 2^{-c} + \frac{1}{3} = 0$ for $c=0$ and $c=1$ and is strictly positive for $c\geq 2$. Thus this expression is always $\geq 0$ and we have the desired inequality.

\uline{Case II: }$\left\langle b,m,P_{2}\right\rangle \in W_{\ell_{b},P_{2}}$
and $\left\langle n-b,m,P_{2}\right\rangle \in W_{\ell_{n-b},P_{2}}$
with $\ell_{b}=\floor{\log_{2}(b-1)}$ and $\ell_{n-b}=\floor{\log_{2}(n-b-1)}$.

We have in this case:
\begin{eqnarray*}
 &  & 1-\frac{b}{n}q_{W_{k}}(m,b)-\frac{n-b}{n}q_{W_{k}}(m,n-b)\\
 & = & 1-\frac{b}{n}\left(\frac{2^{\ell_{b}+1}}{m}-\frac{2}{3}\frac{2^{\ell_{b}+1}-1}{mb}\right)-\frac{n-b}{n}\left(\frac{2^{\ell_{n-b}+1}}{m}-\frac{2}{3}\frac{2^{\ell_{n-b}+1}-1}{m(n-b)}\right)\\
 & = & 1+\frac{4}{3nm}-\frac{2}{nm}\left(\left(b-\frac{2}{3}2^{\ell_b}\right)2^{\ell_b}+\left((n-b)-\frac{2}{3}2^{\ell_{n-b}}\right)2^{\ell_{n-b}}\right).
\end{eqnarray*}
So it satisfies us to \uline{minimize} the function that appears:
\begin{equation*}
\left(b-\frac{2}{3}2^{\ell_{b}}\right)2^{\ell_{b}}+\left((n-b)-\frac{2}{3}2^{\ell_{n-b}}\right)2^{\ell_{n-b}}.
\end{equation*}
Since $\ell_{b}=\floor{\log_{2}(b-1)}$, $\ell_{n-b}=\floor{\log_{2}(n-b-1)}$
we recognize this as $f(b)+f(n-b)$ where $f:\bN\to\bR$ is the function
$f(a)=\left(a-\frac{2}{3}2^{\floor{\log_{2}(a-1)}}\right)2^{\floor{\log_{2}(a-1)}}$.
But now we observe that this function is convex. Moreover this is
a piecewise linear function and can be written as a maximum over a
number of linear functions $f(a)=\sup_{i\in\bN}\left(a-\frac{2}{3}2^{i}\right)2^{i}$.
Observe then that $f(b)+f(n-b)\geq f(\floor{\frac{1}{2}n})+f\left(\ceil{\frac{1}{2}n}\right)$.
In other words the bid $b=\floor{\frac{1}{2}n}$ cannot be beaten. When
we plug in $b=\floor{\frac{1}{2}n}$ we get exactly the desired bound
of $q_{U_{k}}(n,m)$. 

\uline{Case III:} $\left\langle b,m,P_{2}\right\rangle \in U_{k,P_{2}}$
and $\left\langle n-b,m,P_{2}\right\rangle \in U_{k,P_{2}}$ 

There are no such values of $b$ since we would have $b\geq2^{k}+1$ and
$n-b\geq2^{k}+1$ from the definition of $U_{k,P_2}$, which by summing is seen to contradict $n\leq2^{k+1}$ from the fact that $\left\langle n,m,P_{1} \right\rangle \in U_{k,P_{1}}$.\end{proof}
\begin{rem}
A careful reading of Case II of Lemma \ref{lem:qUpperHand} reveals
that for some values of $n$, there are other bids $b$ that match
the performance of $b=\floor{\half n}$. Indeed, the function $f$
is piecewise linear so when $\floor{\half n}$ is on the interior
of a linear segment, other bids that have both $b$ and $n-b$ on
the same segment will do just as well. \end{rem}
\begin{thm}
(Restatement of Theorem \ref{thm:Main}) We have that $p^{\star}(n,m)=q(n,m)$
for all $n,m$ and an optimal bidding is $b^{\star}(n,m)=2^{k}$ if
$\left\langle n,m,P_{1}\right\rangle \in W_{k,P_{1}}$ and $b^{\star}(n,m)=\floor{\half n}$
when $\left\langle n,m,P_{1}\right\rangle \in U_{k,P_{1}}$.\end{thm}
\begin{proof}
From Lemma \ref{lem:qWeeds} and Lemma \ref{lem:qUpperHand}, we know
that $q$ and $p^{\star}$ both satisfy the same recursion relation
from Proposition \ref{prop:Recursion}. We now prove that $p^{\star}(n,m)=q(n,m)$
for all $(n,m)$ by induction on the sum $n+m$. Let $A_{s}=\left\{ (n,m):n+m=s\right\} $
and consider the statement that $p^{\star}(n,m)=q(n,m)$ for every for all
$(n,m)\in\cup_{3\leq s\leq r}A_{s}$. The base case $r=3$ is clear
since $p^{\star}(2,1)=q(2,1)=0$ and $p^{\star}(1,2)=q(1,2)=1$. Now suppose $p^{\star}(n,m)=q(n,m)$
for every for all $(n,m)\in A_{r}$ $\forall3\leq s\leq r$. We have
then for any $(n,m)\in A_{s}$ that:
\begin{eqnarray*}
p^{\star}(n,m) & = & \max_{b\in[1,n-1]}\left\{1-\frac{b}{n}p^{\star}(m,b)-\frac{n-b}{n}p^{\star}(m,n-b)\right\}\text{ (recurrence for }p^{\star})\\
 & = & \max_{b\in[1,n-1]}\left\{1-\frac{b}{n}q(m,b)-\frac{n-b}{n}q(m,n-b)\right\}\text{ (by induction hypothesis)}\\
 & = & q(n,m)\text{ (recurrence relation for }q).
\end{eqnarray*}
The induction hypothesis can be applied since every pair $\left(m,b\right)$,
$(m,n-b)$ has $m+b\leq s-1$ and $m+n-b\leq s-1$ since $1\leq b\leq n-1$.
By the principle of induction, this shows that $p^{\star}(n,m)=q(n,m)$ for
all $(n,m)\in\cup_{3\leq s\leq r}A_{s}$ and for all $r\geq3$. But
$\cup_{r\geq3}A_{r}$ covers the entire quadrant. Thus $p^{\star}(n,m)=q(n,m)$
for all $n,m$ as desired. In light of $p^{\star}\equiv q$, Lemma
\ref{lem:qWeeds} and Lemma \ref{lem:qUpperHand} also show that $b^{\star}(n,m)=2^{k}$
for $\left\langle n,m,P_{1}\right\rangle \in W_{k,P_{1}}$ and $b^{\star}(n,m)=\floor{\half n}$
when $\left\langle n,m,P_{2}\right\rangle \in U_{k,P_{1}}$ is an
optimal bidding strategy.
\end{proof}

\section*{Acknowledgments}
The author was supported by a MacCracken fellowship from New York University and National Science Foundation grant DMS-1209165. The author is also very grateful to the anonymous referee for their careful reading of the article and in particular for suggestions that greatly improved the presentation of ideas in Section 3.

\bibliographystyle{apt}
\bibliography{GuessWhoBib}

\inputencoding{latin1}Email:\inputencoding{latin9}\foreignlanguage{english}{\href{mailto:nica@cims.nyu.edu}{nica@cims.nyu.edu}}\selectlanguage{english}%

\end{document}